\documentclass[11pt]{article}
\usepackage{amssymb}
\usepackage{amsmath,amsthm}
\usepackage{latexsym}
\usepackage{amsfonts}
\usepackage{graphicx}

\newtheorem{theorem}{Theorem}

\newtheorem{corollary}[theorem]{Corollary}

\newtheorem{definition}[theorem]{Definition}
\newtheorem{example}[theorem]{Example}

\newtheorem{proposition}[theorem]{Proposition}

\begin{document}

\title{Legendre curves and singularities of a ruled surface according to
rotation minimizing frame}
\author{\textsc{Murat Bekar}$^{1\text{,}}$\thanks{%
Corresponding author. \ \ \ \ \ \ \ \ \ \ \ \ \ \ \ \ \ \ \ \ \ \ \ \ \ \ \
\ \ \ \ \ \ \ \ \ \ \ \ \ \ \ \ \ \ \ \ \ \ \ \ \ \ \ \ \ \ \ \ \ \ \ \ \ \
\ \ \ \ \ \ \ \ \ \ \ \ \ \ \ \ \ \ \ \ \ \ \ \ \ \ \ \ \ \ \ \ \ \ \ \ \ \
\ \ \ \ \ \ \ \ \ \ \ \ \ \ \ \ \ \ \ \ \ \ \ \ \ \ \ \ \ \ \ \ \ \ \ \ \ \
\ \ \ \ \ \ \ \ \ \ \ \ \ \ \ \ \ \ \ \ \ \ \ \ \ \ \ \ \ \ \ \ \ \ \ \ \ \
\ \ \ \ \ \ \ \ \ \ \ \ \ \ \ \ \ \ \ \ \ \ \ \ \ \ \ \ \ \ \ \ \ \ \ \ \ \
\ \ \ \ \ \ \ \ \ \ \ \ \ \ \ \ \ \ \ \ \ \ \ \ \ \ \ \ \ \ \ \ \ \ \ \ \ \
\ \ \ \ \ \ \ \ \ \ \ \ \ \ \ \ \ \ \ \ \ \ \ \ \ \ \ \ \ \ \ \ \ \ \ \ \ \
\ \ \ \ \ \ \ \ \ \ \ \ \ \ \ \ \ \ \ \ \ \ \ \ \ \ \ \ \ \ \ \ \ \ \ \ \ \
\ \ \ \ \ \ \ \ \ \ \ \ \ \ \ \ \ \ \ \ \ \ \ \ \ \ \ \ \ \ \ \ \ \ \ \ \ \
\ \ \ \ \ \ \ \ \ \ \ \ \ \ \ \ \ \ \ \ \ \ \ \ \ \ \ \ \ \ \ \ \ \ \ \ \ \
\ \ \ \ \ \ \ \ \ \ \ \ \ \ \ \ \ \ \ \ \ \ \ \ \ \ \ \ \ \ \ \ \ \ \ \ \ \
\ \ \ \ \ \ \ \ \ \ \ \ \ \ \ \ \ \ \ \ \ \ \ \ \ \ \ \ \ \ \ \ \ \ \ \ \ \
\ \ \ \ \ \ \ \ \ \ 2010 2010 \textit{AMS Mathematics Subject Classification:%
} Primary: 53A04, 32S25; Secondary: 14J60.\ \ \ \ \ \ \ \ \ \ \ \ \ \ \ \ \
\ \ \ \ \ \ \ \ \ \ \ \ \ \ \ \ \ \ \ \ \ \ \ \ \ \ \ \ \ \ \ \ \ \ \ \ \ \
\ \ \ \ \ \ \ \ \ \ \ \ \ \ \ \ \ \ \ \ \ \ \ \ \ \ \ \ \ \ \ \ \ \ \ \ \ \
\ \ \ \ \ \ \ \ \ \ \ \ \ \ \ \ \ \ \ \ \ \ \ \ \ \ \ \ \ \ \ \ \ \ \ \ \ \
\ \ \ \ \ \ \ \ \ \ \ \ \ \ \ \ \ \ \ \ \ \ \ \ \ \ \ \ \ \ \ \ \ \ \ \ \ \
\ \ \ \ \ \ \ \ \ \ \ \ \ \ \ \ \ \ \ \ \ \ \ \ \ \ \ \ \ \ \ \ \ \ \ \ \ \
\ \ \ \ \ \ \ \ \ \ \ \ \ \ \ \ \ \ \ \ \ \ \ \ \ \ \ \ \ \ \ \ \ \ \ \ \ \
\ \ \ \ \ \ \ \ \ \ \ \ \ \ \ \ \ \ \ \ \ \ \ \ \ \ \ \ \ \ \ \ \ \ \ \ \ \
\ \ \ \ \ \ \ \ \ \ \ \ \ \ \ \ \ \ \ \ \ \ \ \ \ \ \ \ \ \ \ \ \ \ \ \ \ \
\ \ \ \ \ \ \ \ \ \ \ \ \ \ \ \ \ \ \ \ \ \ \ \ \ \ \ \ \ \ \ \ \ \ \ \ \ \
\ \ \ \ \ \ \ \ \ \ \ \ \ \ \ \ \ \ \ \ \ \ \ \ \ \ \ \ \ \ \ \ \ \ \ \ \ \
\ \ \ \ \ \ \ \ \ \ \ \ \ \ \ \ \ \ \ \ \ \ \ \ \ \ \ \ \ \ \ \ \ \ \ \ \ \
\ \ \ \ \ \ \ \ \ \ \ \ \ \ \ \ \ \ \ \ \ \ \ \ \ \ \ \ \ \ \ \ \ \ \ \ \ \
\ } , \textsc{Fouzi Hathout}$^{2}$, \textsc{Yusuf Yayli}$^{3}$ \and $^{1}$%
{\small Gazi University, Department of Mathematics, 06900 Polatli/Ankara, }%
\textsc{Turkey} \and {\small murat-bekar@hotmail.com } \and $^{2}${\small %
Saida University, Department of Mathematics, 2000 Saida, }\textsc{Algeria}
\and {\small f.hathout@gmail.com } \and $^{3}${\small Ankara University,
Department of Mathematics, 06100 Ankara, }\textsc{Turkey} \and {\small %
yayli@science.ankara.edu.tr }}
\date{}
\maketitle

\begin{abstract}
In this paper, Legendre curves on unit tangent bundle are given using
rotation minimizing (RM) vector fields. Ruled surfaces corresponding to
these curves are represented. Singularities of these ruled surfaces are also
analyzed and classified.

\end{abstract}

\section{Introduction}

One of the most known orthonormal frame on a space curve is the
Frenet-Serret frame, comprising the tangent vector field $T$, the principal
normal vector field $N$ and the binormal vector field $B=T\times N$. When
this frame is used to orient a body along a path, its angular velocity
vector (known also as the Darboux vector) $W$ satisfies $<W,N>=0$, i.e. it
has no component in the principal normal vector direction. This means that
the body exhibits no instantaneous rotation about the unit normal vector $N$
from point to point along the path.

\medskip

Bishop introduced the rotation minimizing frame (RMF) which is an
alternative to Frenet-Serret frame, see \cite{b}. This alternative frame
does not have an instantaneous rotation about the unit tangent vector field $%
T$. Nowadays, RMF is widely used in mathematical researches and Computer
Aided Geometric Desing, e.g. \cite{a, f, m}.

\medskip

More precisely, in $n$-dimensional Riemannian manifold $\left(
M,g=<,>\right) $, an RMF along a curve $\gamma $ is an orthonormal frame
defined by the tangent vector field $T$ (of the curve $\gamma $ in $M$) and
by $n-1$ normal vector fields $N_{i}$, which do not rotate with respect to
the tangent vector field (i.e., $\nabla _{T}N_{i}$ is proportional to $%
T=\gamma ^{\prime }(s)$, where $\nabla $ is the Levi Civita connection of g$%
) $. This type of a normal vector field along a curve is said to be a
rotation minimizing vector field (RM vector field). Any orthonormal basis $%
\{T(s_{0}),N_{1}(s_{0}),..,N_{n-1}(s_{0})\}$ at a point $\gamma (s_{0})$
defines a unique RMF along the curve $\gamma $. The RMF can be defined at
any situation of the derivatives of the curve $\gamma $. The notion of RMF
particularizes to that of Bishop in Euclidian case, see \cite{e}. The Frenet
type equations of the RMF is given by%
\begin{equation*}
\nabla _{T}T(s)=\overset{n-1}{\underset{i=1}{\sum }}\kappa _{i}(s)N_{i}(s)\
\ \text{and}\ \ \nabla _{T}N_{i}(s)=\kappa _{i}(s)T(s)\text{,}
\end{equation*}%
where $\kappa _{i}(s)$ are called the natural curvatures along the curve $%
\gamma .$

\medskip

On the other hand, Legendre curves (especially in the tangent bundle of $2$%
-sphere, $T\mathbb{S}^{2}$) are studied by many authors, e.g., \cite{h1, h2}%
. We call the pair $\Gamma =(\gamma ,v)\subset T\mathbb{S}^{2}$ satisfying $%
<\gamma ^{\prime },v>=0$ as Legendre curve. We have proved that any two RM
vector fields correspond to a Legendre curve in (the unit tangent bundle of $%
2$-sphere) $UT\mathbb{S}^{2}$ of some curves, see Theorems \ref{101} and \ref%
{102}.

\medskip

In \cite{h2}, we have shown that to any Legendre curve in $T\mathbb{S}^{2}$
corresponds a developable ruled surface. According to RMF along a curve in
3-dimensional manifold, one can define six ruled surfaces. In this study, we
want to describe what the offsetting process does to the local shape of a
curve. In particular, we want to determine what happens to the singularities
on the ruled surfaces which we have considered. We have observed that our
six ruled surfaces can be one of the following according to their
singularities: \textit{Cuspidal edge} $C\times \mathbb{R}$, \textit{%
Swallowtail }SW, \textit{Cuspidal crosscap} CCR or a \textit{cone surface}.

\medskip

This paper is divided into two parts: In Section 2, we give some definitions
and notions about the Legendre curves in $UT\mathbb{S}^{2}$ and about the RM
vector fields. By Theorems \ref{101} and \ref{102}, we give some
relationships between these curves and vector fields. In Section 3, we show
that the ruled surfaces obtained from RMF are developable and we analyze the
singularities of these ruled surfaces.

\bigskip

All curves and manifolds considered in this paper are of class $C^{\infty }$
unless otherwise stated.

\section{Legendre curves and RM vectors fields}

Let $\gamma :I\subset \mathbb{R}\rightarrow M$ be a non-null curve with
arc-length parameter $s$ in three-dimensional Riemannian manifold $(M,g=$ $%
<,>)$. Then, there exists an accompanying three-frame $\left \{
T,N,B\right
\} $ known as the \textit{Frenet-Serret frame} of $\gamma
=\gamma (s)$. In this case, the moving Frenet-Serret formulas in $M$ are
given by
\begin{equation}
\left(
\begin{array}{c}
\nabla _{T}T(s) \\
\nabla _{T}N(s) \\
\nabla _{T}B(s)%
\end{array}%
\right) =\left(
\begin{array}{ccc}
0 & \kappa (s) & 0 \\
-\kappa (s) & 0 & \tau (s) \\
0 & -\tau (s) & 0%
\end{array}%
\right) \left(
\begin{array}{c}
T(s) \\
N(s) \\
B(s)%
\end{array}%
\right) ,  \label{1}
\end{equation}%
where $\kappa (s)$ and $\tau (s)$ are called the \textit{curvature} and the
\textit{torsion} of the curve $\gamma $ at the point $s$, respectively. The
set $\left \{ T,N,B,\kappa ,\tau \right \} $ is also called the \textit{%
Frenet-frame apparatus}.

\begin{definition}
Let $\gamma $ be a curve in $(M,g)$. A normal vector field $N$ over $\gamma $
is said to be a rotation minimizing vector field (RM vector field) if it is
parallel with respect to the normal connection of $\gamma $. This means that
$\nabla _{\gamma ^{\prime }}N$ and $\gamma ^{\prime }$ are proportional.
\end{definition}

A rotation minimizing frame (RMF) along a curve $\gamma =\gamma (s)$ in $%
(M^{3},g)$ is an orthonormal frame defined by tangent vector $T$ and by two
normal vector fields $N_{1}$ and $N_{2}$, which are proportional to $T$. Any
orthonormal basis $\left \{ T,N_{1},N_{2}\right \} $ at a point $\gamma
(s_{0}) $ defines a unique RMF along the curve $\gamma $. Let $\nabla $ be
the Levi Civita connection of the metric $g$. Then, Frenet type equations
read as
\begin{equation}
\left(
\begin{array}{c}
\nabla _{T}T(s) \\
\nabla _{T}N_{1}(s) \\
\nabla _{T}N_{2}(s)%
\end{array}%
\right) =\left(
\begin{array}{ccc}
0 & \kappa _{1}(s) & \kappa _{2}(s) \\
-\kappa _{1}(s) & 0 & 0 \\
-\kappa _{2}(s) & 0 & 0%
\end{array}%
\right) \left(
\begin{array}{c}
T(s) \\
N_{1}(s) \\
N_{2}(s)%
\end{array}%
\right) .  \label{2}
\end{equation}%
Here, the functions $\kappa _{1}(s)$ and $\kappa _{2}(s)$ are called the
\textit{natural curvatures} of RMF given by%
\begin{equation*}
\kappa (s)=\sqrt{\kappa _{1}^{2}(s)+\kappa _{2}^{2}(s)}\ \ \text{and \ \ }%
\tau (s)=\theta ^{\prime }(s)=\frac{\kappa _{1}(s)\kappa _{2}^{\prime
}(s)-\kappa _{1}^{\prime }(s)\kappa _{2}(s)}{\kappa _{1}^{2}(s)+\kappa
_{2}^{2}(s)}\text{,}
\end{equation*}%
where $\theta (s)=\arg (\kappa _{1}(s),\kappa _{2}(s))=\arctan \frac{\kappa
_{2}(s)}{\kappa _{1}(s)}$ and $\theta ^{\prime }(s)$ is the derivative of $%
\theta (s)$ with respect to the arc-length.

\bigskip

If $(M,g)$ is the Euclidean 3-space $(\mathbb{R}^{3},<,>)$, then the notion
of RMF particularizes to that of Bishop frame.

\bigskip

Let $\mathbb{S}^{2}$ be the unit 2-sphere in $\mathbb{R}^{3}$. Then, the
tangent bundle of $\mathbb{S}^{2}$ is given by
\begin{equation*}
T\mathbb{S}^{2}=\{(\gamma ,v)\in \mathbb{R}^{3}\times \mathbb{R}^{3}\text{ }:%
\text{ }\left\vert \gamma \right\vert =1\ \text{and }<\gamma ,v>=0\}
\end{equation*}%
and the unit tangent bundle of $\mathbb{S}^{2}$ is given by
\begin{eqnarray}
UT\mathbb{S}^{2} &=&\{(\gamma ,v)\in \mathbb{R}^{3}\times \mathbb{R}^{3}%
\text{ }:\text{ }\left\vert \gamma \right\vert =\left\vert v\right\vert =1\
\text{and }<\gamma ,v>=0\} \\
&=&\{(\gamma ,v)\in \mathbb{S}^{2}\times \mathbb{S}^{2}\text{ }:\text{ }%
<\gamma ,v>=0\}  \notag
\end{eqnarray}%
which is a $3$-dimensional contact manifold and its canonical contact $1$%
-form is $\theta ,$ where $<,>$ and $\left\vert ,\right\vert $ denotes the
usual \textit{inner product} and \textit{norm} in $\mathbb{R}^{3}$,
respectively. For further information see \cite{h1, t}

\bigskip

In general, in any Riemannian manifold a curve $\gamma $ is said to be
\textit{Legendre} if it is an integral curve of the contact distribution $%
D=\ker \theta $, i.e. $\theta (\gamma ^{\prime })=0$, see \cite{bb}. In
particular, Legendre curves in 3-dimensional contact manifold $UT\mathbb{S}%
^{2}$ on $\mathbb{S}^{2}$ can be given by the following definition:

\begin{definition}
The smooth curve
\begin{equation*}
\Gamma (s)=(\gamma (s),v(s)):I\subset \mathbb{R}\rightarrow UT\mathbb{S}%
^{2}\subset \mathbb{S}^{2}\times \mathbb{S}^{2}
\end{equation*}%
is called \textit{Legendre curve} in $UT\mathbb{S}^{2}$ if
\begin{equation}
<\gamma ^{\prime }(s),v(s)>=0\text{.}  \label{3}
\end{equation}
\end{definition}

The Legendre curve condition in $UT\mathbb{S}^{2}$ can be seen in \cite{hd}
as a definition of $\Delta $-dual to each other in $\mathbb{S}^{2}$. By the
following theorem we give the relationship between RM vector fields and the
Legendre curve conditions in $UT\mathbb{S}^{2}$:

\begin{theorem}
\label{101}Let $\gamma :I\subset \mathbb{R}\rightarrow \mathbb{S}^{2}$ be a
regular unit speed curve with the frame apparatus $\left \{ T,N,B,\kappa
,\tau \right \} $. Then, we have the following assertions:
\end{theorem}

\begin{enumerate}
\item \textit{If }$N_{1}(s)$\textit{\ and }$N_{2}(s)$\textit{\ are RM vector
fields along }$\gamma $\textit{, the curve }$(N_{1}(s),N_{2}(s))$\textit{\
is Legendre in} $UT\mathbb{S}^{2}$.

\item \textit{If }$\overline{N}_{1}(s)$\textit{\ and }$\overline{N}_{2}(s)$%
\textit{\ are RM vectors along }$B$\textit{-direction curve }$\overline{%
\beta }(s)=\int B(s)ds,$\textit{\ the curve }$(\overline{N}_{1}(s),\overline{%
N}_{2}(s))$\textit{\ is Legendre in }$UTS^{2}$\textit{.}

\item \textit{If }$B(s)$\textit{\ and }$T(s)$\textit{\ are RM vector fields
along }$N$\textit{-direction curve }$\beta (s)=\int N(s)ds,$\textit{\ the
curve }$(B(s),T(s))$\textit{\ is Legendre in }$UTS^{2}$\textit{.}
\end{enumerate}

\begin{proof}
Assume that $\gamma :I\subset \mathbb{R}\rightarrow \mathbb{S}^{2}$ is a
regular unit speed curve with the frame apparatus $\left \{ T,N,B,\kappa
,\tau \right \} $. Then,

\begin{enumerate}
\item Consider the curve $\Gamma (s)=(N_{1}(s),N_{2}(s))\in UT\mathbb{S}^{2}$%
. Since $N_{1}(s)$ and $N_{2}(s)$ are RM vector fields along $\gamma (s)$,
from Equation (\ref{2}) we get that
\begin{equation*}
<N_{1}^{\prime }(s),N_{2}(s)>=-\kappa _{1}(s)<T(s),N_{2}(s)>=0\text{.}
\end{equation*}%
Thus, from Equation (\ref{3}) we can say that $\Gamma $ is a Legendre curve
in $UT\mathbb{S}^{2}.$

\item Consider the curve $\Gamma (s)=(\overline{N}_{1}(s),\overline{N}%
_{2}(s))\in UT\mathbb{S}^{2}$ along the $B$-direction curve $\overline{\beta
}(s)$. The Frenet type equations can be given as
\begin{equation}
\left(
\begin{array}{c}
B^{\prime }(s) \\
\overline{N}_{1}^{\prime }(s) \\
N_{2}^{\prime }(s)%
\end{array}%
\right) =\left(
\begin{array}{ccc}
0 & \bar{\kappa}_{1}(s) & \bar{\kappa}_{2}(s) \\
-\bar{\kappa}_{1}(s) & 0 & 0 \\
-\bar{\kappa}_{2}(s) & 0 & 0%
\end{array}%
\right) \left(
\begin{array}{c}
B(s) \\
\overline{N}_{1}(s) \\
\overline{N}_{2}(s)%
\end{array}%
\right)  \label{4}
\end{equation}%
with the natural curvatures%
\begin{equation*}
\bar{\kappa}(s)=\sqrt{\bar{\kappa}_{1}^{2}(s)+\bar{\kappa}_{2}^{2}(s)}\  \  \
\text{and \ }\bar{\tau}(s)=\theta ^{\prime }(s)=\frac{\bar{\kappa}%
_{1}^{\prime }(s)\bar{\kappa}_{2}(s)-\bar{\kappa}_{1}^{\prime }(s)\bar{\kappa%
}_{2}(s)}{\bar{\kappa}_{1}^{2}(s)+\bar{\kappa}_{2}^{2}(s)}\text{.}
\end{equation*}%
From Equation (\ref{4}), we get that
\begin{equation*}
<\overline{N}_{1}^{\prime }(s),\overline{N}_{2}(s)>=-\bar{\kappa}%
_{1}(s)<B(s),\overline{N}_{2}(s)>=0\text{.}
\end{equation*}%
Thus, from Equation (\ref{3}), we can say that $\Gamma $ is a Legendre curve
in $UT\mathbb{S}^{2}.$ The proof of Assertion 3 can be given by the similar
way as Assertions 1 and 2.
\end{enumerate}
\end{proof}

From the definition of the set $UT\mathbb{S}^{2}$, we know that for a smooth
curve $\Gamma (s)=\left( \gamma (s),v(s)\right) $ in $T\mathbb{S}^{2}$ it is
$<\gamma (s),v(s)>=0$. Thus, we can define a new frame using the unit vector
$\eta (s)=\gamma (s)\wedge v(s)$, where $\wedge $ denotes the usual \textit{%
vector product} in $\mathbb{R}^{3}.$ It is obvious that $<\gamma (s),\eta
(s)>$ $=$ $<v(s),\eta (s)>=0$. Hence, we get the following Frenet frame $\{
\gamma (s),v(s),\eta (s)\}$ along $\gamma (s)$;%
\begin{equation}
\left(
\begin{array}{c}
\gamma ^{\prime }(s) \\
v^{\prime }(s) \\
\eta ^{\prime }(s)%
\end{array}%
\right) =\left(
\begin{array}{ccc}
0 & l(s) & m(s) \\
-l(s) & 0 & n(s) \\
-m(s) & -n(s) & 0%
\end{array}%
\right) \left(
\begin{array}{c}
\gamma (s) \\
v(s) \\
\eta (s)%
\end{array}%
\right) \text{,}  \label{5}
\end{equation}%
where $l(s)=<\gamma ^{\prime }(s),v(s)>,\ m(s)=<\gamma ^{\prime }(s),\mu
(s)> $, $n(s)=<v^{\prime }(s),\mu (s)>.$ The triple $\left \{ l,m,n\right \}
$ is called the \textit{curvature functions }of $\Gamma $.

\bigskip

We know that, if $l(s)=0,$the curve $\Gamma (s)=(\gamma (s),v(s))$ is
Legendre in $UT\mathbb{S}^{2}$ with the curvature functions $\left(
m,n\right) $.

\begin{theorem}
\label{102}Let $\Gamma (s)=(\gamma (s),v(s))$ be a smooth curve in $UT%
\mathbb{S}^{2}$. If $\Gamma (s)$ is Legendre, the vectors $\gamma (s)$ and $%
v(s)$ are RM vector fields along the $\eta $-direction curve $\beta $, i.e. $%
\beta (s)=\int \eta (s)ds),$ and the triple vector field set $\left \{
\gamma ,v,\eta \right \} $ is an RMF.
\end{theorem}

\begin{proof}
Let $\Gamma (s)=(\gamma (s),v(s))$ be a smooth Legendre curve in $UT\mathbb{S%
}^{2}.$ Then, the frenet frame Equation (\ref{5}\textbf{)} for Legendre
condition $($that is, $l(s)=0)$ can be given by%
\begin{equation}
\left(
\begin{array}{c}
\eta ^{\prime }(s) \\
\gamma ^{\prime }(s) \\
v^{\prime }(s)%
\end{array}%
\right) =\left(
\begin{array}{ccc}
0 & -m(s) & -n(s) \\
m(s) & 0 & 0 \\
n(s) & 0 & 0%
\end{array}%
\right) \left(
\begin{array}{c}
\eta (s) \\
\gamma (s) \\
v(s)%
\end{array}%
\right) \text{.}  \label{6}
\end{equation}%
From Equation (\ref{2}\textbf{)}, we can say that $\left \{ \eta ,\gamma
,v\right \} $ is an RMF along the $\eta $-direction curve $\beta (s)=\int
\eta (s)ds$.
\end{proof}

\section{Singularities of ruled surface according to RMF}

A \textit{ruled surface} in $\mathbb{R}^{3}$ is locally the map%
\begin{equation*}
\Phi _{\left( \beta ,\alpha \right) }:I\times \mathbb{R\longrightarrow R}^{3}
\end{equation*}%
defined by%
\begin{equation*}
\Phi _{\left( \beta ,\alpha \right) }(s,u)=\beta (s)+u\alpha (s)\text{,}
\end{equation*}%
where $\beta $ and $\alpha $ are smooth mappings defined from an open
interval $I$ (or a unit circle $\mathbb{S}^{1}$) to $\mathbb{R}^{3}$. $\beta
$ is the \textit{base curve} (or \textit{directrix}) and the non-null curve $%
\alpha $ is the \textit{director curve. }The straight lines\textit{\ }$u%
\mathbb{\longrightarrow }\beta (s)+u\alpha (s)$ are the \textit{rulings.}

\bigskip

The \textit{striction curve} of the ruled surface $\Phi _{\left( \beta
,\alpha \right) }(s,u)=\beta (s)+u\alpha (s)$ is defined by%
\begin{equation}
\bar{\beta}(s)=\beta (s)-\frac{\left \langle \beta ^{\prime }(s),\alpha
^{\prime }(s)\right \rangle }{\left \langle \alpha ^{\prime }(s),\alpha
^{\prime }(s)\right \rangle }\alpha (s).
\end{equation}%
If $\left \langle \beta ^{\prime }(s),\alpha ^{\prime }(s)\right \rangle $ $%
=0, $ the striction curve $\bar{\beta}(s)$\ coincides with the base curve $%
\beta (s)$.

\bigskip

A ruled surface $\Phi _{\left( \beta ,\alpha \right) }(s,u)=\beta
(s)+u\alpha (s)$ is said to be \textit{developable} if%
\begin{equation*}
\det \left( \beta ^{\prime }(s),\alpha (s),\alpha ^{\prime }(s)\right) =0%
\text{.}
\end{equation*}

From Theorem \ref{102}, we can say that if $\Gamma $ is a Legendre curve,
the vector set $\left \{ \eta ,\gamma ,v\right \} $ is an RMF along the $%
\eta $-direction curve $\beta (s)=\int \eta (s)ds$. One can define by this
frame the following six ruled surfaces:
\begin{equation}
\Phi _{_{\left( a_{1i},a_{2i}\right) }}(s,u)=a_{1i}(s)+u_{i}a_{2i}(s);\ \ \
\text{\ for \ }i=1,\text{..., }6  \label{8}
\end{equation}%
where $a_{1i}(s)$ and $a_{2i}(s)$ are different unit curves from the set $%
\left \{ \beta (s),\gamma (s),v(s)\right \} .$

\begin{proposition}
\label{103}Ruled surfaces $\Phi _{_{\left( a_{1i},a_{2i}\right) }}(s,u)$,
for $i=1,...,$ $6$, given by Equation (\ref{8}) are developable.
\end{proposition}

\begin{proof}
Let $\Phi _{_{\left( a_{11},a_{21}\right) }}(s,u)=\beta (s)+u\gamma (s)$ be
a ruled surface defined by Equation (\ref{8}). Using Equation (\ref{6}), we
get the developability condition of $\Phi _{_{\left( a_{11},a_{21}\right) }}$%
;%
\begin{equation*}
\det (\beta ^{\prime }(s),\gamma (s),\gamma ^{\prime }(s))=\det (\eta
(s),\gamma (s),m(s)\eta (s))=0.
\end{equation*}%
Proof of the other ruled surfaces $\Phi _{_{\left( a_{1i},a_{2i}\right) }}$
for $i=2,...,$ $6$ can be given by the similar way.
\end{proof}

Now, recall the parametric equations of the surfaces \textit{Cuspidal edge},
\textit{Swallowtail }and \textit{Cuspidal crosscap} in $\mathbb{R}^{3}$
given by Figure 1, see \cite{it}:

\begin{enumerate}
\item \textit{Cuspidal edge}: $C\times \mathbb{R}=\left \{ (x_{1},x_{2})%
\text{ };\text{ }x_{1}^{2}=x_{2}^{3}\right \} \times \mathbb{R}$.

\item \textit{Swallowtail}: SW$=\left \{ (x_{1},x_{2},x_{3})\text{ };\text{ }%
x_{1}=3u^{4}+u^{2}v\text{,}\ x_{2}=4u^{3}+2uv\text{, }x_{3}=v\right \} $.

\item \textit{Cuspidal crosscap}: CCR$=\left\{ (x_{1},x_{2},x_{3})\text{ };%
\text{ }x_{1}=u^{3}\text{, \ }x_{2}=u^{3}v^{3}\text{, }\ x_{3}=v^{2}\right\}
$.
\end{enumerate}

\begin{figure}[h]
\centering
\includegraphics[width=4.4815in,height=1.20377in]{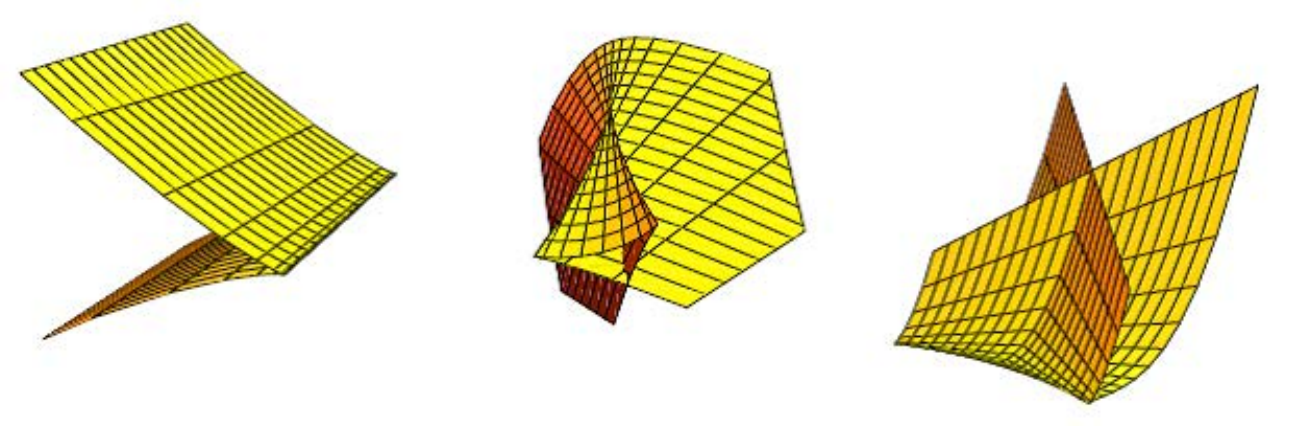}
\caption{\textmd{Left surface is the \textit{Cuspidal edge} $C\times \mathbb{%
R}$, middle surface is the \textit{Swallowtail }SW and right surface is the
\textit{Cuspidal crosscap} CCR.}}
\label{}
\end{figure}

\bigskip

By the following theorem, we give the local classification of singularities
of the ruled surfaces defined by using Equation (\ref{8}):

\begin{theorem}
\label{104}Let $\Gamma (s)=(\gamma (s),v(s))$ be a smooth Legendre curve in $%
UT\mathbb{S}^{2}.$ According to RMF $\left \{ \eta ,\gamma ,v\right \} $
along the $\eta $-direction curve $\beta (s)$, we have the following:
\end{theorem}

\begin{enumerate}
\item $\Phi _{\left( \beta ,\gamma \right) }(s,u)=\beta (s)+u\gamma (s)$%
\textit{\ which is locally diffeomorphic to;}

\begin{enumerate}
\item $C\times R\ $\textit{at }$\Phi _{\left( \beta ,\gamma \right)
}(s_{0},u_{0})$\textit{\ if and only if }$u_{0}=-m(s_{0})^{-1}\neq 0$\textit{%
\ and }$m^{\prime }(s_{0})\neq 0.$

\item $SW$\textit{\ at }$\Phi _{\left( \beta ,\gamma \right) }(s_{0},u_{0})$%
\textit{\ if and only if }$u_{0}=-m(s_{0})^{-1}\neq 0$\textit{, }$m^{\prime
}(s_{0})=0$\textit{\ and }$(m(s_{0})^{-1})^{\prime \prime }(s_{0})\neq 0$%
\textit{.}
\end{enumerate}

\item $\Phi _{\left( \beta ,v\right) }(s,u)=\beta (s)+uv(s)$\textit{\ which
is locally diffeomorphic to;}

\begin{enumerate}
\item $C\times R\ $\textit{\ at }$\Phi _{\left( \beta ,v\right)
}(s_{0},u_{0})$\textit{\ if and only if }$u_{0}=-n(s_{0})^{-1}\neq 0$\textit{%
\ and }$u^{\prime }(s_{0})\neq 0.$

\item $SW$\textit{\ at }$\Phi _{\left( \beta ,v\right) }(s_{0},u_{0})$%
\textit{\ if and only if }$u_{0}=-n(s_{0})^{-1}\neq 0$\textit{, }$n^{\prime
}(s_{0})=0$\textit{\ and }$(n(s_{0})^{-1})^{\prime \prime }(s_{0})\neq 0$%
\textit{.}
\end{enumerate}

\item $\Phi _{\left( \beta ,\gamma \right) }(s,u)=\beta (s)+u\gamma (s)$%
\textit{\ (resp. }$\Phi _{\left( \beta ,v\right) }(s,u)=\beta (s)+uv(s)$%
\textit{) which is a cone surface if and only if }$m(s)$\textit{\ (resp., }$%
n(s)$\textit{) is constant.}
\end{enumerate}

\begin{proof}
Assume that $\Gamma (s)=(\gamma (s),v(s))$ is a smooth Legendre curve in $UT%
\mathbb{S}^{2}$ according to the RMF $\left \{ \eta ,\gamma ,v\right \} $
along the $\eta $-direction curve $\beta (s).$ Using Equation (\ref{8}) and $%
\Phi _{\left( \beta ,\gamma \right) }(s,u)=\beta (s)+u\gamma (s)$, we get%
\begin{eqnarray*}
\frac{\partial \Phi _{\left( \beta ,\gamma \right) }}{\partial s}(s,u)
&=&(1+u\text{ }m(s))\eta \text{,} \\
\frac{\partial \Phi _{\left( \beta ,\gamma \right) }}{\partial u}(s,u)
&=&\gamma \text{,} \\
\frac{\partial \Phi _{\left( \beta ,\gamma \right) }}{\partial s}(s,u)\wedge
\frac{\partial \Phi _{\left( \beta ,\gamma \right) }}{\partial u}(s,u)
&=&(1+u\text{ }m(s))v\text{.}
\end{eqnarray*}%
Singularities of the normal vector field of $\Phi _{\left( \beta ,\gamma
\right) }=\Phi _{\left( \beta ,\gamma \right) }(s,u)$ are%
\begin{equation*}
u=\frac{-1}{m(s)}\text{.}
\end{equation*}%
From Theorem 3.3 of the paper \cite{it}, we know that if there exists a
parameter $s_{0}$ such that $u_{0}=\frac{-1}{m(s_{0})}\neq 0$ and $%
u_{0}^{\prime }=\frac{m^{\prime }(s_{0})}{m^{2}(s_{0})}\neq 0$ (i.e., $%
m^{\prime }(s_{0})\neq 0$), then $\Phi (s,u)$ is locally diffeomorphic to
the $C\times \mathbb{R}\ $\ at $\Phi _{\left( \beta ,\gamma \right)
}(s_{0},u_{0})$. This completes the proof of Assertion 1.(a). Again from
Theorem 3.3 of \cite{it}, we know that if there exists a parameter $s_{0}$
such that $u_{0}=\frac{-1}{m(s_{0})}\neq 0$, $u_{0}^{\prime }=\frac{%
m^{\prime }(s_{0})}{m^{2}(s_{0})}=0$ and $(m(s_{0})^{-1})^{\prime \prime
}(s_{0})\neq 0$, then $\Phi _{\left( \beta ,\gamma \right) }$ is locally
diffeomorphic to $SW$ at $\Phi _{\left( \beta ,\gamma \right) }(s_{0},u_{0})$%
, and this completes the proof of Assertion 1.(b).\newline
The proof of Assertion 2 can be given similar to the proof of Assertion 1.
The proof of Assertion 3 can be given as: The singularity points are equal
to the striction curve of $\Phi $ and can be given by%
\begin{equation*}
\varphi _{\left( \beta ,\gamma \right) }(s)=\Phi _{\left( \beta ,\gamma
\right) }(s,\frac{-1}{m(s)})=\beta (s)-\frac{1}{m(s)}\gamma (s)
\end{equation*}

\begin{equation*}
\left( \text{resp., }\varphi _{\left( \beta ,v\right) }(s)=\Phi _{\left(
\beta ,v\right) }(s,\frac{-1}{m(s)})=\beta (s)-\frac{1}{m(s)}v(s)\right)
\text{.}
\end{equation*}%
Thus, we have%
\begin{equation*}
\varphi _{\left( \beta ,\gamma \right) }^{\prime }(s)=-(\frac{1}{m(s)}%
)^{\prime }\gamma (s)\text{ \  \ }\left( \text{resp, }\varphi _{\left( \beta
,v\right) }^{\prime }(s)=-(\frac{1}{m(s)})^{\prime }v(s)\right)
\end{equation*}%
which means that if $m(s)$ is a constant function, then%
\begin{equation*}
\varphi _{\left( \beta ,\gamma \right) }^{\prime }(s)=\varphi _{\left( \beta
,v\right) }^{\prime }(s)=0.
\end{equation*}%
So, $\Phi _{\left( \beta ,\gamma \right) }$ (resp., $\Phi _{\left( \beta
,v\right) }$) has only one singularity point and thus it is a cone surface.
\end{proof}

\begin{corollary}
\label{c1}Let $\alpha :I\subset \mathbb{R}\rightarrow \mathbb{R}^{3}$ be a
smooth curve with frame apparatus $\left \{ N_{1},N_{2},\kappa _{1},\kappa
_{2}\right \} $ given by Equation (\ref{4}). If we choose $\Gamma
(s)=(\gamma (s),v(s))=\Gamma (N_{1}(s),N_{2}(s)),$ we obtain the Theorem 3.1
given in \cite{hd}. If we choose $\Gamma (s)=(\gamma ,v)=\Gamma
(N_{2}(s),N_{1}(s))$, we obtain the Theorem 3.2 given in \cite{hd}$.$
\end{corollary}

\begin{proof}
Let $\alpha :I\subset \mathbb{R}\rightarrow \mathbb{R}^{3}$ be a smooth
curve with frame apparatus $\left \{ N_{1},N_{2},\kappa _{1},\kappa
_{2}\right \} $ given by Equation (\ref{2}). The vector fields $\left \{
T,N_{1},N_{2}\right \} $ is an RMF along the $T$-direction curve $\beta
(s)=\alpha (s)=\int T(s)ds$. This means that $\Gamma (N_{1}(s),N_{2}(s))$ is
a Legendre curve in $T\mathbb{S}^{2}$. Using Theorem \ref{104}, we complete
the proof, where $m(s)=\kappa _{1}(s)$ and $n(s)=\kappa _{2}(s).$
\end{proof}

\begin{theorem}
\label{105}Let $\Gamma (s)=(\gamma (s),v(s))$ be a smooth Legendre curve in $%
UT\mathbb{S}^{2}.$Then, according to RMF $\left \{ \eta ,\gamma ,v\right \} $
along the $\eta $-direction curve $\beta (s)$, we have the foolowing:
\end{theorem}

\begin{enumerate}
\item $\Phi _{\left( \gamma ,\beta \right) }(s,u)=\gamma (s)+u\beta (s)$%
\textit{\ which is locally diffeomorphic to;}

\begin{enumerate}
\item $C\times R\ $\textit{at }$\Phi _{\left( \gamma ,\beta \right)
}(s_{0},u_{0})$\textit{\ if and only if }$u_{0}=-m(s_{0})\neq 0$\textit{\
and }$m^{\prime }(s_{0})\neq 0.$

\item $SW$\textit{\ at }$\Phi _{\left( \gamma ,\beta \right) }(s_{0},u_{0})$%
\textit{\ if and only if }$u_{0}=-m(s_{0})\neq 0$\textit{, }$m^{\prime
}(s_{0})=0$\textit{\ and }$m^{\prime \prime }(s_{0})\neq 0$\textit{.}

\item $CCR$\textit{\ at }$\Phi _{\left( \gamma ,\beta \right) }(s_{0},u_{0})$%
\textit{\ if and only if }$u_{0}=-m(s_{0})=0$\textit{\ and }$m^{\prime
}(s_{0})\neq 0$\textit{.}
\end{enumerate}

\item $\Phi _{\left( v,\beta \right) }(s,u)=v(s)+u\beta (s)$\textit{\ which
is locally diffeomorphic to;}

\begin{enumerate}
\item $C\times R\ $\textit{\ at }$\Phi _{\left( v,\beta \right)
}(s_{0},u_{0})$\textit{\ if and only if }$u_{0}=-n(s_{0})\neq 0$\textit{\
and }$n^{\prime }(s_{0})\neq 0.$

\item $SW$\textit{\ at }$\Phi _{\left( v,\beta \right) }(s_{0},u_{0})$%
\textit{\ if and only if }$u_{0}=-n(s_{0}),$\textit{\ }$n^{\prime }(s_{0})=0$%
\textit{\ and }$n^{\prime \prime }(s_{0})\neq 0$\textit{.}

\item $CCR$\textit{\ at }$\Phi _{\left( v,\beta \right) }(s_{0},u_{0})$%
\textit{\ if and only if }$u_{0}=-n(s_{0})=0$\textit{\ and }$n^{\prime
}(s_{0})\neq 0$\textit{.}
\end{enumerate}

\item $\Phi _{\left( \gamma ,\beta \right) }(s,u)=\gamma (s)+u\beta (s)$%
\textit{\ (resp., }$\Phi _{\left( v,\beta \right) }(s,u)=v(s)+u\beta (s)$%
\textit{) which is a cone surface if and only if }$m(s)$\textit{\ (resp., }$%
n(s)$\textit{) is constant.}
\end{enumerate}

\bigskip Proofs of Theorems \ref{105} and \ref{106} can be given similar to
the proof of Theorem \ref{104}.

\begin{theorem}
\label{106}Let $\Gamma (s)=(\gamma (s),v(s))$ be a smooth Legendre curve in $%
UT\mathbb{S}^{2}$ with curvature functions $\left \{ m,n\right \} $. Then we
have the following:
\end{theorem}

\begin{enumerate}
\item \textit{Ruled surface }$\Phi _{\left( \gamma ,v\right) }(s,u)=\gamma
(s)+uv(s)$\textit{\ is locally diffeomorphic to;}

\begin{enumerate}
\item $C\times R\ $\textit{\ at }$\Phi _{\left( \gamma ,v\right)
}(s_{0},u_{0})$\textit{\ if and only if }$u_{0}=-\frac{m}{n}(s_{0})\neq 0$%
\textit{\ and }$(\frac{m}{n})^{\prime }(s_{0})\neq 0.$

\item $SW$\textit{\ at }$\Phi _{\left( \gamma ,v\right) }(s_{0},u_{0})$%
\textit{\ if and only if }$u_{0}=-\frac{m}{n}(s_{0})\neq 0$\textit{, }$(%
\frac{m}{n})^{\prime }(s_{0})=0$\textit{\ and }$(\frac{m}{n})^{\prime \prime
}(s_{0})\neq 0$\textit{.}

\item $CCR$\textit{\ at }$\Phi _{\left( \gamma ,v\right) }(s_{0},u_{0})$%
\textit{\ if and only if }$u_{0}=-\frac{m}{n}(s_{0})=0$\textit{\ (i.e., }$%
m(s_{0})=0$ \textit{and} $n(s_{0})\neq 0$\textit{)\ and }$(\frac{m}{n}%
)^{\prime }(s_{0})\neq 0$\textit{.}
\end{enumerate}

\item \textit{Ruled surface }$\Phi _{\left( v,\gamma \right)
}(s,u)=v(s)+u\gamma (s)$\textit{\ is locally diffeomorphic to;}

\begin{enumerate}
\item $C\times R\ $\textit{\ at }$\Phi _{\left( v,\gamma \right)
}(s_{0},u_{0})$\textit{\ if and only if }$u_{0}=-\frac{n}{m}(s_{0})\neq 0$%
\textit{\ and }$(\frac{n}{m})^{\prime }(s_{0})\neq 0.$

\item $SW$ \textit{at }$\Phi _{\left( v,\gamma \right) }(s_{0},u_{0})$%
\textit{\ if and only if }$u_{0}=-\frac{n}{m}(s_{0})\neq 0,$\textit{\ }$(%
\frac{n}{m})^{\prime }(s_{0})=0$\textit{\ and }$(\frac{n}{m})^{\prime \prime
}(s_{0})\neq 0$

\item $CCR$\textit{\ at }$\Phi _{\left( v,\gamma \right) }(s_{0},u_{0})$%
\textit{\ if and only if }$u_{0}=-\frac{n}{m}(s_{0})=0$\textit{\ (i.e., }$%
n(s_{0})=0$ \textit{and} $m(s_{0})\neq 0$\textit{)\ and }$(\frac{n}{m}%
)^{\prime }(s_{0})\neq 0$\textit{.}
\end{enumerate}

\item \textit{Ruled surfaces }$\Phi _{\left( \gamma ,v\right) }(s,u)=\gamma
(s)+uv(s)$\textit{\ (resp., }$\Phi _{\left( v,\gamma \right)
}(s,u)=v(s)+u\gamma (s)$\textit{) is a cone surface if and only if }$\frac{n%
}{m}(s)$\textit{\ (resp., }$\frac{m}{n}(s)$\textit{) is constant.}
\end{enumerate}

\begin{corollary}
\label{c2}Let $\alpha :I\subset \mathbb{R}\rightarrow \mathbb{R}^{3}$ be a
smooth curve with frame apparatus $\left \{ T,N,B,\kappa ,\tau \right \} $.
If we choose $\Gamma (s)=(\gamma (s),v(s))=\Gamma (B(s),T(s)),$ we obtain
the Theorem 3.2 given in \cite{it}.
\end{corollary}

\begin{proof}
Since $T$ and $B$ are RM vector fields along the $T$-direction curve $\beta
(s)=\alpha (s)=\int T(s)ds$, the curve $\Gamma (B(s),T(s))$ is a Legendre in
$T\mathbb{S}^{2}$. Using Theorem 8 and taken $m(s)=\kappa _{1}(s)$, $%
n(s)=\kappa _{2}(s),$ we get the proof.
\end{proof}

\begin{corollary}
\label{c3}Let $\alpha :I\subset \mathbb{R}\rightarrow \mathbb{R}^{3}$ be a
smooth curve with frame apparatus $\left \{ N,C,W=\overline{D},f,g\right \} $%
, see \cite{bg, by}. If we choose $\Gamma (s)=(\gamma (s),v(s))=\Gamma
(W(s),N(s))$, we obtain the Theorem 3.3 given in \cite{it}, where%
\begin{equation*}
W(s)=\overline{D}(s)=\frac{\tau (s)T(s)+\kappa (s)B(s)}{\sqrt{\kappa
^{2}(s)+\tau ^{2}(s)}}
\end{equation*}%
is the unit Darboux vector field.
\end{corollary}

\bigskip

\begin{proof}
Let $\alpha :I\subset \mathbb{R}\rightarrow \mathbb{R}^{3}$ be a smooth
curve with frame apparatus $\left \{ N,C,W=\overline{D},f,g\right \} $. Then,
the curve $\Gamma (s)=(\gamma (s),v(s))=\Gamma (W(s),N(s))$ is a Legendre in
$T\mathbb{S}^{2}$. Using Theorem 9, we get the slant helix condition
\begin{equation*}
\frac{m}{n}(s)=\left( \frac{\kappa ^{2}}{(\kappa ^{2}+\tau ^{2})^{\frac{3}{2}%
}}(\frac{\tau }{\kappa })^{\prime }\right) (s)=\sigma (s)
\end{equation*}%
which completes the proof.
\end{proof}

\bigskip

We close this section by given some examples to illustrate the main results.
The first example is an application of Theorem \ref{106}:

\begin{example}
Let us take a smooth curve $\gamma :I\subset \mathbb{R}\rightarrow \mathbb{R}%
^{3}$ given by%
\begin{equation*}
\gamma (s)=\frac{1}{\sqrt{2}}(-\cos (s),-\sin (s),1)
\end{equation*}%
and a unit vector given by%
\begin{equation*}
v(s)=\frac{1}{\sqrt{2}}(\cos (s),\sin (s),0).
\end{equation*}%
Then, we have%
\begin{equation*}
<\gamma ^{\prime }(s),v(s)>=0\text{.}
\end{equation*}%
Thus, $\Gamma (s)=(\gamma ,v)$ is a Legendre curve in $UT\mathbb{S}^{2}.$
The RMF $\left\{ \eta ,\gamma ,v\right\} $ along the $\eta $-direction curve
$\beta (s)=\int \eta (s)ds$ can be given as%
\begin{equation*}
\left(
\begin{array}{c}
\eta ^{\prime }(s) \\
\gamma ^{\prime }(s) \\
v^{\prime }(s)%
\end{array}%
\right) =\left(
\begin{array}{ccc}
0 & \frac{1}{\sqrt{2}} & -\frac{1}{\sqrt{2}} \\
-\frac{1}{\sqrt{2}} & 0 & 0 \\
\frac{1}{\sqrt{2}} & 0 & 0%
\end{array}%
\right) \left(
\begin{array}{c}
\eta (s) \\
\gamma (s) \\
v(s)%
\end{array}%
\right) \text{.}
\end{equation*}%
The ruled surface
\begin{equation*}
\Phi _{\left( v,\gamma \right) }(s,u)=v(s)+u\gamma (s)=\frac{1}{\sqrt{2}}%
(\cos (s)-u\cos (s),\sin (s)-u\sin (s),u)
\end{equation*}%
represents a cone surface, see Figure 2.
\end{example}

\begin{figure}[h]
\centering
\includegraphics[width=2.7342in,height=2.7296in]{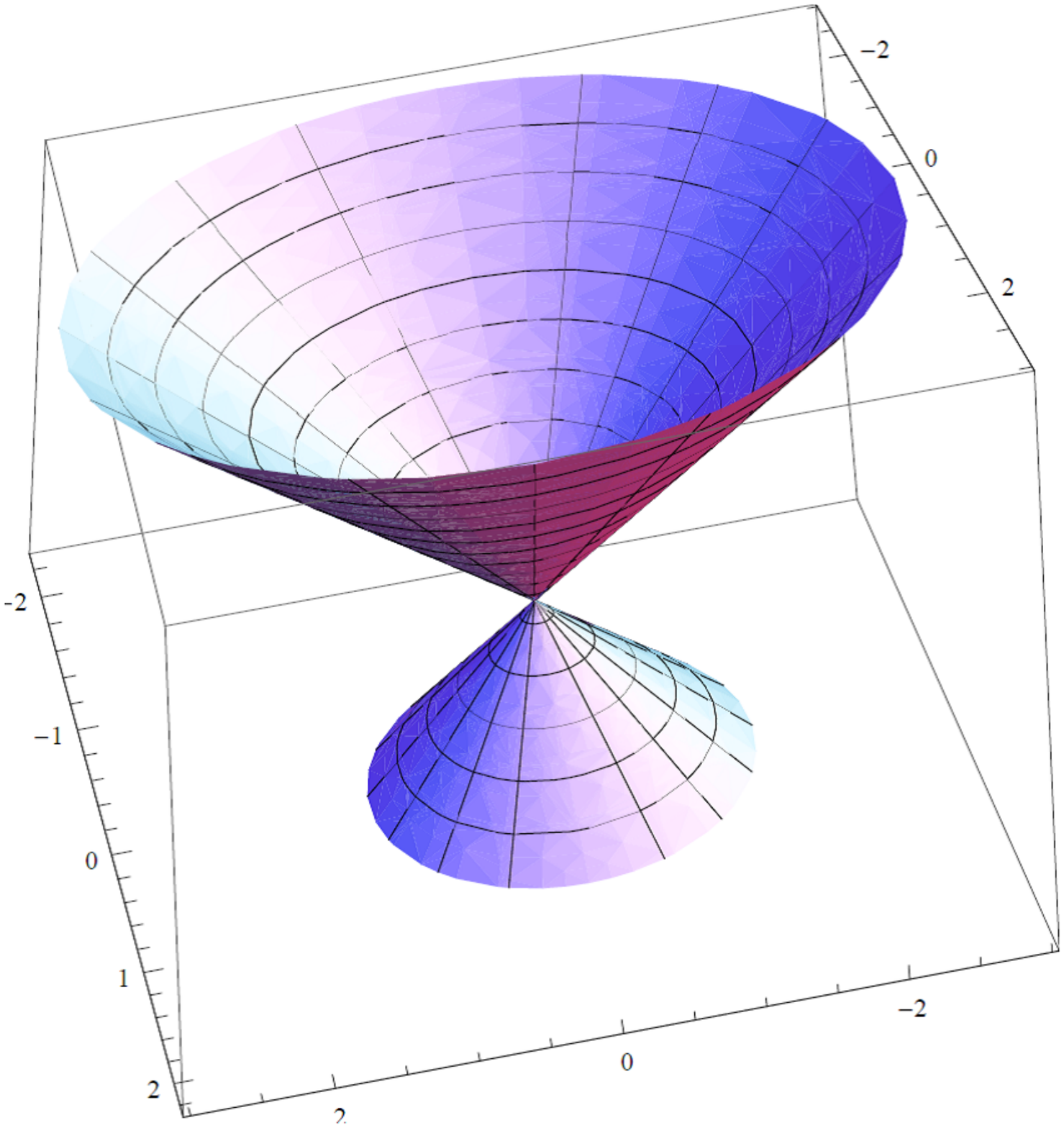}
\caption{The ruled surface $\Phi _{\left( v,\protect\gamma \right)}(s,u)$ is
the cone surface with one singularity point.}
\label{}
\end{figure}

The second example is an application of Theorem \ref{104}:

\begin{example}
Let $\alpha :I\subset \mathbb{R}\rightarrow \mathbb{R}^{3}$ be a smooth
curve defined by%
\begin{equation*}
\gamma (s)=(\cos (\frac{s}{\sqrt{2}}),\sin (\frac{s}{\sqrt{2}}),\frac{s}{%
\sqrt{2}})\text{.}
\end{equation*}%
Then, the tangent and binormal vector fields of $\alpha $\ are, respectively,%
\begin{eqnarray*}
T(s) &=&\frac{1}{\sqrt{2}}(-\sin (\frac{s}{\sqrt{2}}),\cos (\frac{s}{\sqrt{2}%
}),1)\text{,} \\
B(s) &=&\frac{1}{\sqrt{2}}(\sin (\frac{s}{\sqrt{2}}),\cos (\frac{s}{\sqrt{2}}%
),1)
\end{eqnarray*}%
with the curvature $\kappa =\frac{1}{2}$ and the torsion $\tau =\frac{1}{2}$%
. So, $\gamma $ is a helix. The curve $\Gamma (s)=(B,T)$ is Legendre in $UT%
\mathbb{S}^{2}$ and the ruled surface
\begin{eqnarray*}
\Phi _{\left( B,T\right) }(s,u) &=&B(s)+uT(s) \\
&=&\frac{1}{\sqrt{2}}((1-u)\sin (\frac{s}{\sqrt{2}}),(u+1)\cos (\frac{s}{%
\sqrt{2}}),1+u)
\end{eqnarray*}%
is a cone. We get the singularity point for $u=1$ on the point $\Phi
_{\left( B,T\right) }(s,1)=(0,0,\sqrt{2}),$ see Figure 3.

\begin{figure}[h]
\centering
\includegraphics[width=2.7572in,height=2.7333in]{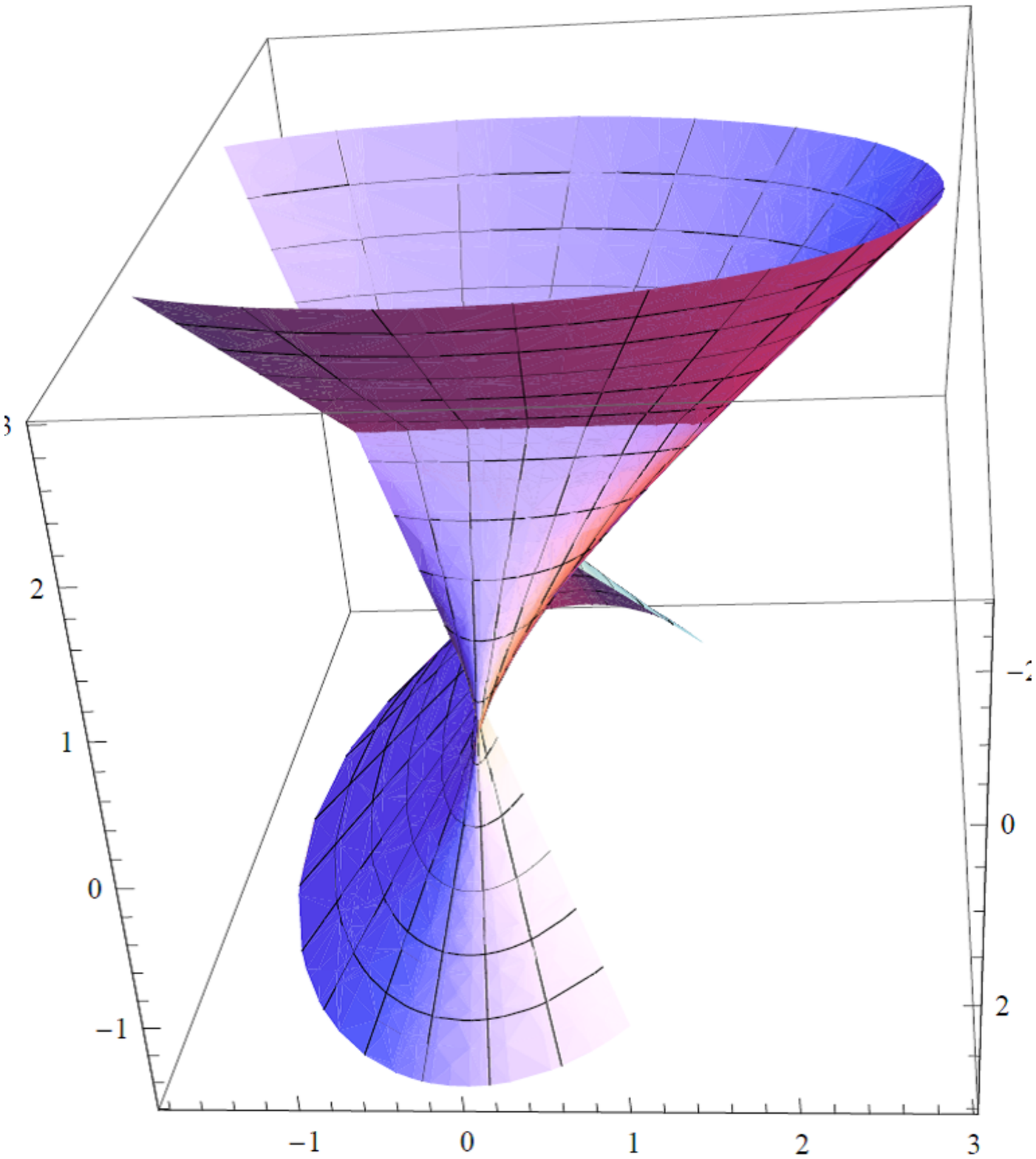}
\caption{The ruled surface $\Phi_{\left( B,T\right) }(s,u)$ is the cone
surface with helix singularity curve.}
\label{}
\end{figure}
\end{example}

The last example is an application of Theorem \ref{105}:

\begin{example}
Let $\alpha :I=\left[ 0,A\right] \rightarrow \mathbb{R}^{3}$ be a smooth
curve (for $0<A\leqslant 2\pi $) defined by%
\begin{eqnarray*}
\gamma (s) &=&\frac{1}{4}(3\cos (s)-\cos (3s),3\sin (s)-\sin (3s),2\sqrt{3}%
\cos (s))\text{,} \\
v(s) &=&\frac{1}{4}(3\sin (s)-\sin (3s),-3\cos (s)-\cos (3s),-2\sqrt{3}\sin
(s))\text{,} \\
\eta (s) &=&\frac{1}{2}(\sqrt{3}\cos (2s),\sqrt{3}\sin (2s),-1)\text{.}
\end{eqnarray*}%
Then, $\Gamma (s)=(\gamma (s),v(s))$ is a Legendre curve with Legendre
curvature function%
\begin{equation*}
m(s)=\sqrt{3}\sin (s)
\end{equation*}%
and we have the following:\newline
1. If $A=\pi $, then $m(\frac{\pi }{2})=\sqrt{3}\neq 0$, $m^{\prime }(\frac{%
\pi }{2})=0$ and $m^{\prime \prime }(\frac{\pi }{2})=-\sqrt{3}\neq 0.$ The
ruled surface
\begin{eqnarray*}
\Phi _{\left( \beta ,\gamma \right) }(s,u) &=&\beta (s)+u\gamma (s) \\
&=&(\frac{\sqrt{3}}{2}\sin (2s)+\frac{3}{4}u\cos (s)-\frac{1}{4}u\cos (3s),
\\
&&-\frac{\sqrt{3}}{2}\cos (2s)-\frac{3}{4}u\sin (s)-\frac{1}{4}u\sin (3s),-%
\frac{s}{2}+\frac{\sqrt{3}}{2}u\cos (s))
\end{eqnarray*}%
is locally diffeomorphic to $C\times \mathbb{R}$ at $\Phi _{\left( \beta
,\gamma \right) }(\frac{\pi }{2},\frac{-1}{\sqrt{3}})$, see Figure 4.
\begin{figure}[h]
\centering
\includegraphics[width=2.5106in,height=2.2in]{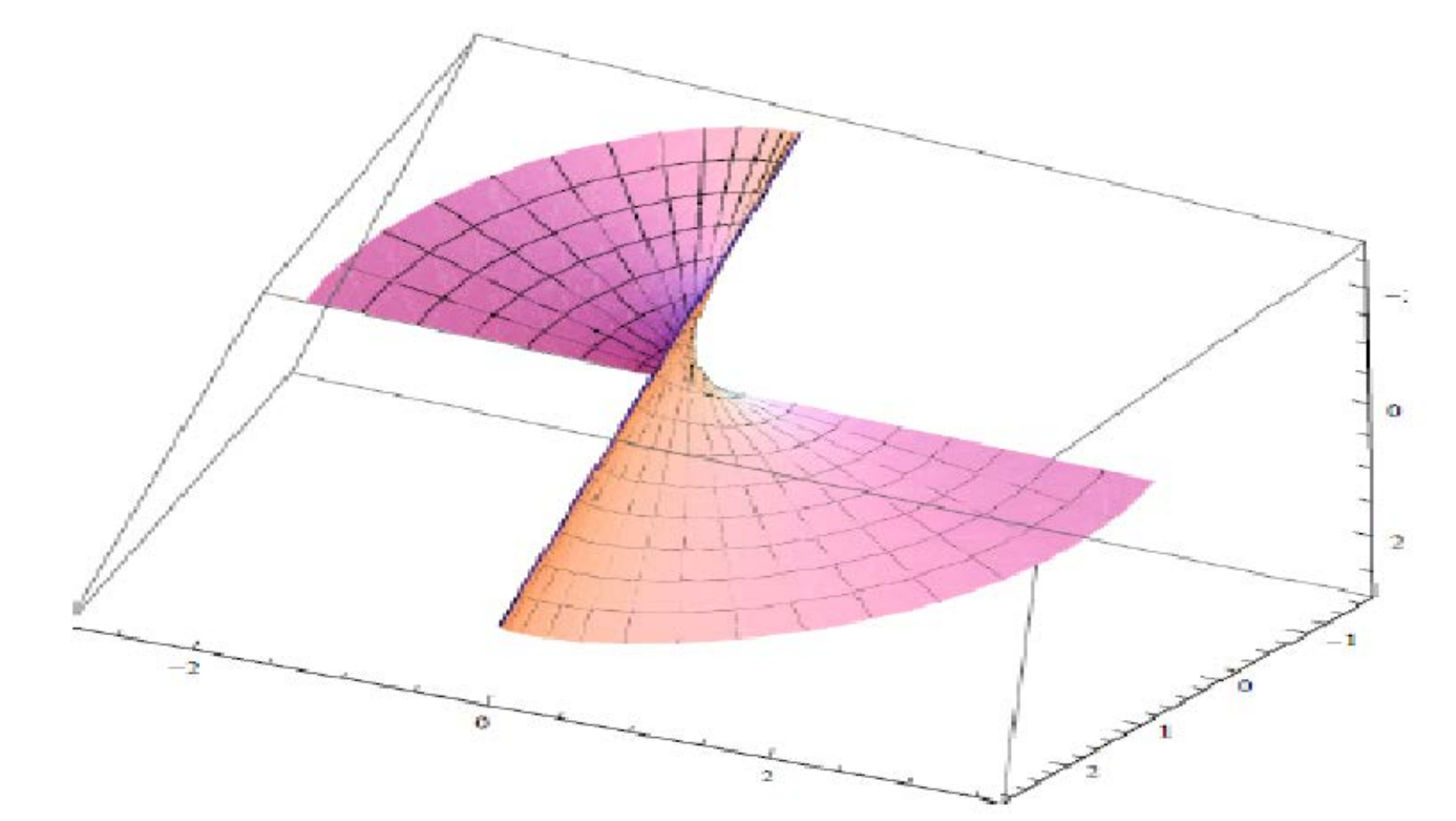}
\caption{\textmd{The cuspidal edge $C\times \mathbb{R}$ at $\Phi _{\left(
\protect\beta ,\protect\gamma \right) }(\frac{\protect\pi }{2},\frac{-1}{%
\protect\sqrt{3}})$}}
\end{figure}

2. If $A=\frac{\pi }{2}$, then $u_{0}=m^{-1}(s_{0})\neq 0$, $%
(m^{-1})^{\prime }(s_{0})\neq 0.$ The ruled surface $\Phi _{\left( \beta
,\gamma \right) }(s,u)$ is locally diffeomorphic to $SW$ at $\Phi _{\left(
\beta ,\gamma \right) }(\frac{\pi }{2},u_{0})$, see Figure 5.

\begin{figure}[h]
\centering
\includegraphics[width=2.6658in,height=2.2in]{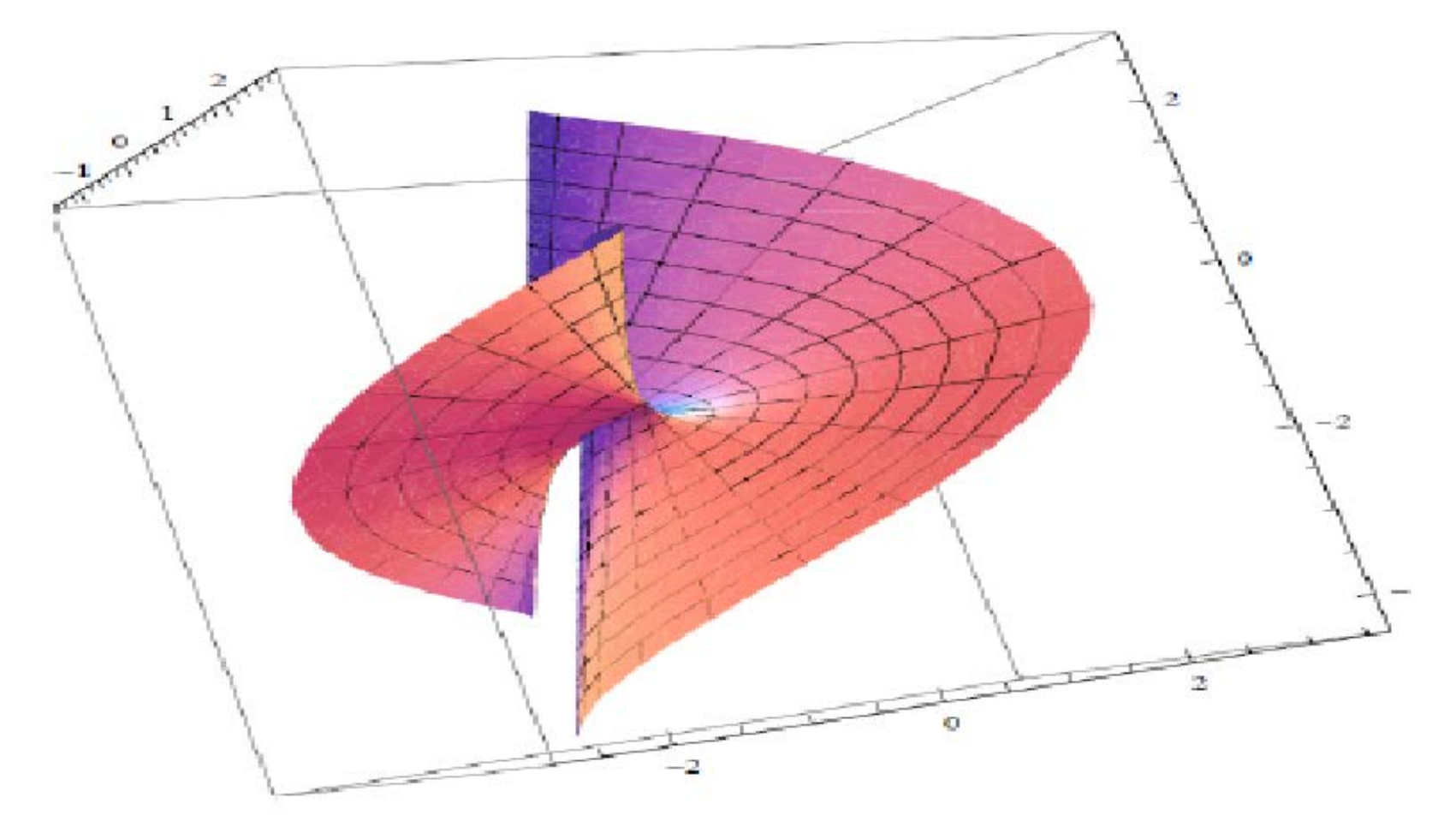}
\caption{\textmd{The swallowtail SW at $\Phi _{\left( \protect\beta ,\protect%
\gamma \right) }(\frac{\protect\pi }{2},u_{0})$}}
\end{figure}
\end{example}

\newpage

\section{Conclusions}

In this paper, we give the Legendre curves on the unit tangent bundle using
the rotation minimizing (RM) vector fields. We represent the ruled surfaces
corresponding to these Legendre curves and discuss their singularities. For
some special cases, given by Corollaries \ref{c1}, \ref{c2} and \ref{c3}, we
get the main ideas of the studies \cite{hd} and \cite{it}.

\end{document}